\documentclass[11pt, letterpaper]{article}
\usepackage{fullpage}

\usepackage{amsmath,amsthm,amssymb,amsfonts}
\usepackage{mathtools}
\usepackage{mathrsfs}
\usepackage[shortlabels]{enumitem}
\usepackage{thmtools}

\usepackage{float}
\usepackage[hidelinks]{hyperref}
\usepackage[nameinlink,capitalise,noabbrev]{cleveref}
\crefformat{equation}{(#2#1#3)}

\usepackage{xcolor}
\usepackage{graphicx}
\usepackage{tikz}
\usetikzlibrary{calc}

\theoremstyle{definition}
\newtheorem{definition}{Definition}

\theoremstyle{plain}

\newtheorem{theorem}[definition]{Theorem}

\newtheorem{claim}[definition]{Claim}

\crefname{claim}{Claim}{Claims}
\crefname{lemma}{Lemma}{Lemmas}

\def \sm {\setminus}
\def \ce {\coloneqq}

\renewcommand{\le}{\leqslant}
\renewcommand{\ge}{\geqslant}

\def \es {\varnothing}
\renewcommand \b[2] {\binom{#1}{#2}}

\def \F{\mathcal{F}}
\def \mH{\mathcal{H}}
\def \bF {\bar{\F}}

\title{On the minimum size of maximal $k$-wise intersecting families}

\author{Haoran Luo\thanks{Department of Mathematics, Statistics and Computer Science, University of Illinois Chicago, Chicago, Illinois 60607, USA. The author was partially supported by an AMS-Simons Travel Grant.
Email: \texttt{haoranl8@uic.edu}.
}}

\date{}

\begin{document}
\maketitle

\begin{abstract}
    A family $\mathcal{F}$ of subsets of $[n] \coloneqq \{1,2,\ldots, n\}$ is called maximal $k$-wise intersecting if every collection of at most $k$ members of $\mathcal{F}$ has a non-empty intersection, and adding any other set to $\mathcal{F}$ breaks this property.
    An old question by Erd\H{o}s and Kleitman from 1974 asks for the minimum size of a maximal $k$-wise intersecting family.
    The case $k = 3$ is known for all sufficiently large $n$, but the problem remains open for all $k \geqslant 4$.
    The previous best-known upper bound is by Janzer, which has a leading term $(k-1)2^{k-3}2^{n/(k-1)}$ for sufficiently large $n$ divisible by $k-1$. In this note, we improve this bound to $(4k-10)2^{n/(k-1)}$, which reduces the dependence on $k$ in the leading coefficient from exponential to linear and is within a factor of $4$ of the known lower bound. 
\end{abstract}

\section{Introduction} \label{sec::Int}
The intersection property lies in a central position in the study of extremal set theory.
For integers $n, k \ge 2$, a family $\F$ of subsets of $[n] \ce \{1,2,\ldots, n\}$ is \emph{$k$-wise intersecting} if every collection of at most $k$ members in $\F$ has a non-empty intersection. A well-known fact is that every ($2$-wise) intersecting family has size at most $2^{n-1}$, which is achieved by the family of sets containing a fixed element, 
and hence, every $k$-wise intersecting family also has the same upper bound of size. Intersection problems for families of sets have been studied extensively, see Frankl and Tokushige~\cite{frankl2016invitation} for a survey.

A family $\F$ of subsets of $[n]$ is \emph{maximal $k$-wise intersecting} if it is $k$-wise intersecting and adding any other set to $\F$ breaks this property.
Let $f_k(n)$ be the minimum possible size of a maximal $k$-wise intersecting family on the ground set $[n]$.
An old question by Erd\H{o}s and Kleitman~\cite{erdos1974extremal} from 1974 asks to determine $f_k(n)$.
The fact $f_2(n) = 2^{n-1}$ follows from the simple observation that every maximal intersecting family has size exactly $2^{n-1}$.
For the first non-trivial case $k=3$, Balogh, Chen, Hendrey, Lund, Luo, Tompkins, and Tran~\cite{balogh2023maximal} solved it completely for sufficiently large $n$. They proved that the unique extremal family is obtained by partitioning the ground set into two sets $A, B$ whose sizes differ by at most one and then taking the family of all proper supersets of $A$ and of $B$, so $f_3(n) = 2^{\lceil n /2 \rceil} + 2^{\lfloor n /2 \rfloor} - 3$.

However, for the cases $k \ge 4$, Janzer~\cite{janzer2022saturation} showed that the direct generalization of the construction mentioned above cannot be maximal $k$-wise intersecting. Using this result, Balogh, Chen, Hendrey, Lund, Luo, Tompkins, and Tran~\cite{balogh2023maximal} proved that for every $k \ge 4$, there is $c_k > 0$ such that for all sufficiently large $n$ divisible by $k-1$,
\begin{equation} \label{equ::lowerboundfkn}
f_k(n) \ge (1+c_k)\big( (k-1)2^{\frac{n}{k-1}} -k + 2 \big).
\end{equation}
For the upper bound, Janzer~\cite{janzer2022saturation} improved the previous bound $O_k(2^{n / \lceil k/2 \rceil})$ from~\cite{balogh2023maximal} and proved that for all $n \ge 2(k-1)$ divisible by $k-1$,
\[
f_k(n) \le (k-1)2^{k-3} \cdot 2^{\frac{n}{k-1}} - (2^{k-1}-1)(k-2).
\]

In this paper, we obtain the following upper bound on $f_k(n)$.

\begin{theorem} \label{thm::main}
    For integers $k \ge 4$ and $n \ge (k-1)(k-2)$ divisible by $k-1$, we have
    \[
        f_k(n) \le (4k-10)2^{\frac{n}{k-1}} - 7(k-2).
    \]
\end{theorem}
\noindent
Hence, the dependence on $k$ in the leading coefficient is reduced from exponential to linear, and we are within a factor of $4$ of \cref{equ::lowerboundfkn}.
We remark that it agrees with Janzer’s bound for $k = 4$ and improves it for every $k \ge 5$.

\cref{thm::main} follows from the following general theorem, which works for all $n \ge 2k-2$.

\begin{theorem} \label{thm::mainGeneral}
Let $k\ge 4$ be a fixed integer. Assume that $n, m_1, m_2, \ldots, m_{k-1}$ are positive integers with
\[
    n = (k-1) + \sum_{i=1}^{k-1} m_i.
\]
Then, we have
\[
    f_k(n) \le 4 \sum_{i=1}^{k-2} 2^{m_i} + 2^{k-2}2^{m_{k-1}} - 7(k-2).
\]
\end{theorem}

The remaining part of this paper is organized as follows. In \cref{sec::Con}, we give the proof of \cref{thm::mainGeneral} and deduce \cref{thm::main}. In \cref{sec::concluding}, we give some concluding remarks.

\section{Construction} 
\label{sec::Con}
We first prove \cref{thm::mainGeneral} and then deduce \cref{thm::main}.

\begin{proof}[Proof of \cref{thm::mainGeneral}]
    Note that a family $\F$ on the ground set $[n]$ is $k$-wise intersecting if and only if $\bF \ce \{[n] \sm S : S \in \F\}$ has the property that the union of any collection of at most $k$ members from $\bF$ does not equal $[n]$. We will give the construction for $\bF$ and then $\F$ can be recovered naturally.

    We partition the ground set $[n]$ into
    \[
        [n] = X_1 \cup \ldots \cup X_{k-2} \cup X_{k-1} \cup Y \cup \{z\},
    \]
    where 
    \[
        |X_i| = m_i \,\,\, \textrm{for every $1 \le i \le k-1$} \quad \textrm{and} \quad Y = \{y_1, \ldots, y_{k-2}\}.  
    \]
    For every integer $i$ $(1 \le i \le k-2)$, define
    \[
        \mH_i \ce \{A \cup B : A \subsetneq X_i,\, B \subseteq \{y_i, z\} \} \cup \{X_i\},
    \]
    and for $k-1$, define
    \[
        \mH_{k-1} \ce \{A \cup B : A \subsetneq X_{k-1},\, B \subseteq Y\} \cup \{X_{k-1} \cup B : B \subseteq Y,\, |B| \le k-4\}.
    \]
    Let 
    \[
        \bF \ce \bigcup_{i=1}^{k-1} \mH_i.
    \]
    See \cref{fig::construction_k5} for an illustration of the case $k = 5$.

    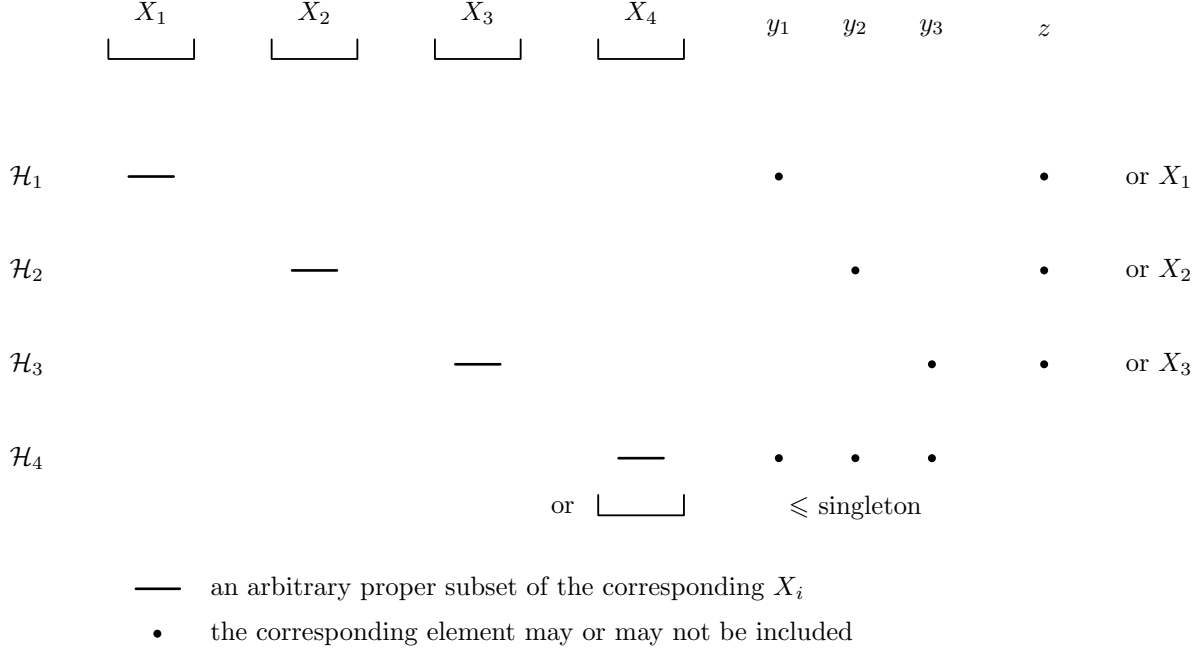
\begin{figure}[ht]
    \centering
    \begin{tikzpicture}[
        x=1.35cm,
        y=1.08cm,
        line cap=round,
        line join=round,
        every node/.style={font=\small}
    ]
    
    \coordinate (X1) at (0,0);
    \coordinate (X2) at (1.60,0);
    \coordinate (X3) at (3.20,0);
    \coordinate (X4) at (4.80,0);
    \coordinate (Y1) at (6.15,0);
    \coordinate (Y2) at (6.90,0);
    \coordinate (Y3) at (7.65,0);
    \coordinate (Z)  at (8.75,0);
    
    \def\ra{-1.30}
    \def\rb{-2.45}
    \def\rc{-3.60}
    \def\rd{-4.75}
    \def\rdb{-5.35}
    
    \def\lga{-6.35}
    \def\lgb{-6.90}
    
    \newcommand{\topblock}[2]{%
        \draw[line width=0.7pt]
            ($(#1)+(-0.42,0.38)$) --
            ($(#1)+(-0.42,0.14)$) --
            ($(#1)+(0.42,0.14)$) --
            ($(#1)+(0.42,0.38)$);
        \node[above=2pt] at ($(#1)+(0,0.38)$) {$#2$};
    }
    
    \newcommand{\fullblock}[2]{%
        \draw[line width=0.7pt]
            ($(#1)+(0,#2)+(-0.42,0.15)$) --
            ($(#1)+(0,#2)+(-0.42,-0.10)$) --
            ($(#1)+(0,#2)+(0.42,-0.10)$) --
            ($(#1)+(0,#2)+(0.42,0.15)$);
    }
    
    \newcommand{\shortsubset}[2]{%
        \draw[line width=1pt]
            ($(#1)+(0,#2)+(-0.22,0)$) --
            ($(#1)+(0,#2)+(0.22,0)$);
    }
    
    \newcommand{\dotat}[2]{%
        \fill ($(#1)+(0,#2)$) circle (1.5pt);
    }
    
    \topblock{X1}{X_1}
    \topblock{X2}{X_2}
    \topblock{X3}{X_3}
    \topblock{X4}{X_4}
    
    \node[above] at ($(Y1)+(0,0.30)$) {$y_1$};
    \node[above] at ($(Y2)+(0,0.30)$) {$y_2$};
    \node[above] at ($(Y3)+(0,0.30)$) {$y_3$};
    \node[above] at ($(Z)+(0,0.30)$) {$z$};
    
    \node[left] at (-0.95,\ra) {$\mH_1$};
    \node[left] at (-0.95,\rb) {$\mH_2$};
    \node[left] at (-0.95,\rc) {$\mH_3$};
    \node[left] at (-0.95,\rd) {$\mH_4$};
    
    \shortsubset{X1}{\ra}
    \dotat{Y1}{\ra}
    \dotat{Z}{\ra}
    \node[anchor=west] at (9.45,\ra) {or $X_1$};
    
    \shortsubset{X2}{\rb}
    \dotat{Y2}{\rb}
    \dotat{Z}{\rb}
    \node[anchor=west] at (9.45,\rb) {or $X_2$};
    
    \shortsubset{X3}{\rc}
    \dotat{Y3}{\rc}
    \dotat{Z}{\rc}
    \node[anchor=west] at (9.45,\rc) {or $X_3$};
    
    \shortsubset{X4}{\rd}
    \dotat{Y1}{\rd}
    \dotat{Y2}{\rd}
    \dotat{Y3}{\rd}
    
    \node[anchor=east] at ($(X4)+(0,\rdb)+(-0.55,0)$) {or};
    \fullblock{X4}{\rdb}
    \node at ($(Y2)+(0,\rdb)$) {$\le$  singleton};
    
    \draw[line width=1pt] (-0.15,\lga) -- (0.28,\lga);
    \node[anchor=west] at (0.48,\lga)
        {an arbitrary proper subset of the corresponding $X_i$};
    
    \fill (0.06,\lgb) circle (1.5pt);
    \node[anchor=west] at (0.48,\lgb)
        {the corresponding element may or may not be included};
    
    \end{tikzpicture}
    
    \caption{An illustration of the construction for $k=5$.
    For $i=1,2,3$, the family $\mH_i$ consists of all sets of the
    form $A\cup B$ with $A\subsetneq X_i$ and
    $B\subseteq\{y_i,z\}$, together with the set $X_i$.
    The family $\mH_4$ consists of all sets of the form $A\cup B$
    with $A\subsetneq X_4$ and
    $B\subseteq Y=\{y_1,y_2,y_3\}$, together with the sets
    $X_4\cup B$ where $B$ is either empty or a singleton.}
    \label{fig::construction_k5}
    \end{figure}

    \begin{claim} \label{cla::sizebF}
        We have
        \[
            |\bF| = 4 \sum_{i=1}^{k-2} 2^{m_i} + 2^{k-2}2^{m_{k-1}} - 7(k-2).
        \]
    \end{claim}
    \begin{proof}
        For every integer $i$ $(1\le i \le k-2)$, we have
        \[
            |\mH_i| =    2^2 \cdot (2^{m_i}-1)+1 = 4 \cdot 2^{m_i} - 3.
        \]
        For integers $i,j$ with $1 \le i < j \le k-2$, we have $\mH_i \cap \mH_j = \{\es, \{z\}\}$. Hence,
        \[
            \Big|\bigcup_{i=1}^{k-2}\mH_i\Big|
            = 
            2 + \sum_{i=1}^{k-2}\big(|\mH_i|-2\big)
            =
            2+\sum_{i=1}^{k-2}(4\cdot2^{m_i}-5)
            =
            4\sum_{i=1}^{k-2}2^{m_i}-5k + 12.
        \]

        For $\mH_{k-1}$, we have
        \begin{align*}
            |\mH_{k-1}| = (2^{m_{k-1}}-1)2^{k-2} + \sum_{j=0}^{k-4}\b{k-2}{j}
            &=
            (2^{m_{k-1}}-1)2^{k-2} + 2^{k-2} - (k-2) - 1 \\
            &= 2^{k-2} 2^{m_{k-1}} - k+1.
        \end{align*}
        Furthermore,
        \[
        \Big(\bigcup_{i=1}^{k-2}\mH_i\Big)
        \cap\mH_{k-1}
        =
        \{\es,\{y_1\},\ldots,\{y_{k-2}\}\},
        \]
        which has size $k-1$. Therefore, we have
        \begin{align*}
            |\bF| =\Big|\bigcup_{i=1}^{k-2}\mH_i\Big|+ |\mH_{k-1}| -(k-1)
            &=
            4\sum_{i=1}^{k-2}2^{m_i}-5k + 12 + \big(2^{k-2} 2^{m_{k-1}} - k+1\big) - (k-1) \\
            &=
            4\sum_{i=1}^{k-2}2^{m_i} + 2^{k-2}2^{m_{k-1}} -7(k-2). \qedhere
        \end{align*}
    \end{proof}

    \begin{claim} \label{cla::kwiseInt}
        For any sets $S_1, \ldots, S_k \in \bF$, we have 
        \[
            \bigcup_{i=1}^k S_i \neq [n].
        \]
    \end{claim}
    \begin{proof}
        Suppose for a contradiction that there are sets $S_1, \ldots, S_k$ covering all elements in $[n]$. Note that in order to cover elements in $X_i$, there must be some index $j$ such that $S_j \in \mH_i$. After relabeling the sets, we may assume that $S_i \in \mH_i$ for every $i$ $(1\le i \le k-1)$. Now, we consider $S_k$.

    Suppose first that $S_k\in\mH_t$ for some $t$ with $1\le t\le k-2$. For every
    $
        i\in\{1,\ldots,k-2\}\sm\{t\},
    $
    the set $S_i$ is the only set among $S_1,\ldots,S_k$ which intersects $X_i$. Hence, $S_i$ must contain all elements in $X_i$. By the definition of $\mH_i$, we have $S_i=X_i$.
    Similarly, $S_{k-1}$ is the only set among $S_1,\ldots,S_k$ which intersects $X_{k-1}$. Hence, it
    contains all of $X_{k-1}$, and so
    $
        S_{k-1}=X_{k-1}\cup B
    $
    for some $B\subseteq Y$ with $|B|\le k-4$. Now, we have that $S_t$ and $S_k$ both belong to $\mH_t$, so
    $
        (S_t\cup S_k)\cap Y\subseteq\{y_t\},
    $
    and all other sets $S_i$ with $1\le i\le k-2$ are equal to
    $X_i$. Therefore,
    \[
        \Big|
        \Big(\bigcup_{i=1}^kS_i\Big)\cap Y
        \Big|
        \le |B|+1
        \le k-3 < k-2 = |Y|,
    \]
    so we have a contradiction that some element in $Y$ is not covered.

    Suppose that $S_k \in \mH_{k-1}$. For every index $i$ with $1 \le i \le k-2$, we have that $S_i$ is the only set among $S_1,\ldots,S_k$ which intersects $X_i$, so $S_i=X_i$; none of them contains $z$. Also, neither
    $S_{k-1}$ nor $S_k$ contains $z$, since every member of $\mH_{k-1}$ is contained in $X_{k-1}\cup Y$. Thus, $z$ is not covered by $S_1, \ldots, S_k$, which is again a contradiction.
    \end{proof}

    \begin{claim} \label{cla::maximal}
        For any set $T \notin \bF$, there are sets $S_1, \ldots, S_{k-1} \in \bF$ such that
        \[
            T \cup \bigcup_{i=1}^{k-1} S_i = [n].
        \]
    \end{claim}
    \begin{proof}
    It suffices to prove the claim for inclusion-minimal sets
    $T\notin\bF$. We first classify all possible forms of such
    sets $T$. We claim that $T$ is of one of the following six
    types.
    \begin{enumerate}[(1)]
        \item $T=\{x_i,x_j\}$, where $x_i\in X_i$,
        $x_j\in X_j$, and $1\le i<j\le k-1$;

        \item $T=\{x_i,y_j\}$, where $x_i\in X_i$,
        $1\le i,j\le k-2$, and $i\ne j$;

        \item $T=X_i\cup\{u\}$, where $1\le i\le k-2$ and
        $u\in\{y_i,z\}$;

        \item $T=\{x_{k-1},z\}$ for some $x_{k-1}\in X_{k-1}$;

        \item $T=X_{k-1}\cup B$, where $B\subseteq Y$ and
        $|B|=k-3$;

        \item $T=\{y_i,y_j,z\}$, where
        $1\le i<j\le k-2$.
    \end{enumerate}

    To see this, suppose first that $T$ intersects two distinct sets
    $X_i$ and $X_j$ where $1 \le i<j \le k-1$. Choose $x_i\in T\cap X_i$ and $x_j\in T\cap X_j$. Since every member of $\bF$ intersects at
    most one of $X_1,\ldots,X_{k-1}$, we have $\{x_i,x_j\}\notin\bF$. 
    By the inclusion-minimality of $T$, we have $T=\{x_i,x_j\}$, so $T$ is of type~(1).

    We may now assume that $T$ intersects at most one of
    $X_1,\ldots,X_{k-1}$. Exactly one of the following three
    cases holds:
    \[
        T\cap X_i\ne\es
        \ \textrm{for some $i$ $(1\le i\le k-2)$},\qquad
        T\cap X_{k-1}\ne\es,\qquad\textrm{or}\qquad
        T\cap\Big(\bigcup_{i=1}^{k-1}X_i\Big)=\es.
    \]

    Suppose that $T\cap X_i\ne\es$ for some integer $i$ with $1\le i\le k-2$. If $y_j\in T$ for some integer $j$
    with $1\le j\le k-2$ and $j\ne i$, choose
    $x_i\in T\cap X_i$. Since
    $\{x_i,y_j\}\notin\bF$, by the inclusion-minimality of $T$,
    we have $T=\{x_i,y_j\}$, so $T$ is of type~(2).
    We may therefore assume that $y_j\notin T$ for every
    $j\ne i$. Since $T$ does not intersect any other $X_j$, we
    have $T\subseteq X_i\cup\{y_i,z\}$. Since
    $T\notin\mH_i$, we must have $X_i\subseteq T$ and
    $T\cap\{y_i,z\}\ne\es$. Choose
    $u\in T\cap\{y_i,z\}$. Since
    $X_i\cup\{u\}\notin\bF$, by the inclusion-minimality of $T$,
    we have $T=X_i\cup\{u\}$, so $T$ is of type~(3).

    Suppose next that $T\cap X_{k-1}\ne\es$. If $z\in T$,
    choose $x_{k-1}\in T\cap X_{k-1}$. Since
    $\{x_{k-1},z\}\notin\bF$, by the inclusion-minimality of
    $T$, we have $T=\{x_{k-1},z\}$, so $T$ is of type~(4).
    We may therefore assume that $z\notin T$. Since $T$ does
    not intersect any of $X_1,\ldots,X_{k-2}$, we have
    $T\subseteq X_{k-1}\cup Y$. Since
    $T\notin\mH_{k-1}$, we must have
    $X_{k-1}\subseteq T$ and $|T\cap Y|\ge k-3$. Choose a set
    $B\subseteq T\cap Y$ with $|B|=k-3$. Since
    $X_{k-1}\cup B\notin\bF$, by the inclusion-minimality of
    $T$, we have $T=X_{k-1}\cup B$, so $T$ is of type~(5).

    Suppose that
    $T\cap\big(\bigcup_{i=1}^{k-1}X_i\big)=\es$. Then,
    $T\subseteq Y\cup\{z\}$. We must have $z\in T$, since every
    subset of $Y$ belongs to $\mH_{k-1}$. Also, $T$ must
    contain at least two elements of $Y$. Indeed, otherwise
    $T\subseteq\{y_i,z\}$ for some integer $i$ with
    $1\le i\le k-2$, and hence $T\in\mH_i$.
    Choose integers $i,j$ with $1\le i<j\le k-2$ such that
    $y_i,y_j\in T$. Since
    $\{y_i,y_j,z\}\notin\bF$, by the inclusion-minimality of
    $T$, we have $T=\{y_i,y_j,z\}$, so $T$ is of type~(6).
    This proves the classification.
   
    We now consider the six types separately. We give the choice of $k-1$ sets from $\bF$ such that together with $T$, they cover all elements in $[n]$.

    \medskip
    \noindent\textbf{Type~(1).}
    Suppose first that $j\le k-2$. Consider the following
    $k-1$ sets:
    \[
        (X_i\sm\{x_i\})\cup\{y_i,z\},\quad
        (X_j\sm\{x_j\})\cup\{y_j\},\quad
        X_\ell\ (\ell\in[k-2]\sm\{i,j\}),\quad
        X_{k-1}\cup(Y\sm\{y_i,y_j\}).
    \]
    The first two sets belong to $\mH_i$ and $\mH_j$,
    respectively. Also, the last set belongs to
    $\mH_{k-1}$ since
    $
        |Y\sm\{y_i,y_j\}|=k-4.
    $
    Therefore, all these sets belong to $\bF$. 

    Suppose that $j=k-1$. Consider the following $k-1$ sets:
    \[
        (X_i\sm\{x_i\})\cup\{y_i,z\},\quad
        X_\ell\ (\ell\in[k-2]\sm\{i\}),\quad
        (X_{k-1}\sm\{x_j\})\cup Y.
    \]
    The last set belongs to $\mH_{k-1}$ since
    $X_{k-1}\sm\{x_j\}\subsetneq X_{k-1}$. Thus, all these
    sets belong to $\bF$.

    \medskip
    \noindent\textbf{Type~(2).}
    Consider the following $k-1$ sets:
    \[
        (X_i\sm\{x_i\})\cup\{y_i,z\},\quad
        X_\ell\ (\ell\in[k-2]\sm\{i\}),\quad
        X_{k-1}\cup(Y\sm\{y_i,y_j\}).
    \]
    The first set belongs to $\mH_i$, and the last set belongs
    to $\mH_{k-1}$ since
    $|Y\sm\{y_i,y_j\}|=k-4$.

    \medskip
    \noindent\textbf{Type~(3).}
    Suppose first that $u=y_i$. Choose an integer
    $j\in[k-2]\sm\{i\}$ and consider the following $k-1$ sets:
    \[
        X_\ell\ (\ell\in[k-2]\sm\{i\}),\quad
        \{y_j,z\},\quad
        X_{k-1}\cup(Y\sm\{y_i,y_j\}).
    \]
    We have $\{y_j,z\}\in\mH_j$, and the last set belongs to
    $\mH_{k-1}$ since $|Y\sm\{y_i,y_j\}|=k-4$.

    Suppose that $u=z$. Consider the following $k-1$ sets:
    \[
        X_\ell\ (\ell\in[k-2]\sm\{i\}),\quad
        X_{k-1},\quad Y.
    \]
    Note that $X_{k-1},Y\in\mH_{k-1}$. Thus, all these sets
    belong to $\bF$.

    \medskip
    \noindent\textbf{Type~(4).}
    Consider the following $k-1$ sets:
    \[
        X_1,\ldots,X_{k-2},\quad
        (X_{k-1}\sm\{x_{k-1}\})\cup Y.
    \]
    The last set belongs to $\mH_{k-1}$ since $X_{k-1}\sm\{x_{k-1}\}\subsetneq X_{k-1}$.

    \medskip
    \noindent\textbf{Type~(5).}
    Let $y_i$ be the unique element in $Y\sm B$. Consider the
    following $k-1$ sets:
    \[
        X_1,\ldots,X_{k-2},\quad \{y_i,z\}.
    \]
    Since $\{y_i,z\}\in\mH_i$, all these sets belong to $\bF$.

    \medskip
    \noindent\textbf{Type~(6).}
    Consider the following $k-1$ sets:
    \[
        X_1,\ldots,X_{k-2},\quad
        X_{k-1}\cup(Y\sm\{y_i,y_j\}).
    \]
    The last set belongs to $\mH_{k-1}$ since $|Y\sm\{y_i,y_j\}|=k-4$.
    Therefore, all these sets belong to $\bF$. 
    \qedhere

    \end{proof}

    Finally, by \cref{cla::kwiseInt,cla::maximal}, we have that $\F$ is maximal $k$-wise intersecting. By \cref{cla::sizebF} and the definition of $f_k(n)$, we have
    \[
        f_k(n) \le 4 \sum_{i=1}^{k-2} 2^{m_i} + 2^{k-2}2^{m_{k-1}} - 7(k-2). \qedhere
    \]
\end{proof}

Now, we prove \cref{thm::main}. 
\begin{proof}[Proof of \cref{thm::main}]
    Let
    \[
        m_i = \frac{n}{k-1} \,\,\textrm{for $1\le i\le k-4$}, \quad m_{k-3} = m_{k-2} = \frac{n}{k-1}-1, \quad \textrm{and} \quad m_{k-1} = \frac{n}{k-1}-k+3.
    \]
    By \cref{thm::mainGeneral}, we have
    \begin{align*}
        f_k(n) 
        &\le 4 \sum_{i=1}^{k-2} 2^{m_i} + 2^{k-2}2^{m_{k-1}}- 7(k-2) \\
        &= 4(k-4) 2^{\frac{n}{k-1}} + 4\cdot 2 \cdot 2^{\frac{n}{k-1}-1} + 2^{k-2} 2^{\frac{n}{k-1}-k+3} - 7(k-2) \\
        &= \Big(4(k-4) + 4\cdot2 \cdot \frac{1}{2} + 2^{k-2} \cdot 2^{-k+3} \Big) \cdot 2^{\frac{n}{k-1}} - 7(k-2)\\
        &= (4k-10)2^{\frac{n}{k-1}} - 7(k-2). \qedhere
    \end{align*}
\end{proof}

\section{Concluding remarks} \label{sec::concluding}
The main difference between the construction in
\cref{sec::Con} and that of
Janzer~\cite{janzer2022saturation} is a slight asymmetry:
the block $X_{k-1}$ and the family $\mH_{k-1}$ are treated
differently from the other blocks and families.

By \cref{equ::lowerboundfkn} and \cref{thm::main}, we have that when $k \ge 4$ is fixed and $n$ is divisible by $k-1$ and sufficiently large,
\[
    (1 + \Omega_k(1)) (k-1) 2^{\frac{n}{k-1}}
    \le
    f_k(n)
    \le
    (4k-10) 2^{\frac{n}{k-1}}.
\]
The main task is to narrow the gap between these two bounds. The construction in \cref{sec::Con} seems still to be artificial and appears unlikely to be optimal. Improving the lower bound could be a harder and more interesting task, which seems to require some genuinely new idea to characterize the structure of a small maximal $k$-wise intersecting family.

More broadly, $k$-wise intersecting families have also been
studied under various additional assumptions, including
uniformity and non-triviality~\cite{oneill2021nontrivial},
maximality in the uniform setting~\cite{zhang2023structure},
symmetry and regularity~\cite{ellis2017symmetric,
frankston2018regular}, and biased product
measures~\cite{tokushige2024measure}. 
The methods developed in these settings may also be useful for
understanding small maximal $k$-wise intersecting families and,
in particular, for improving the lower bound on $f_k(n)$.

\section*{Acknowledgments}
The construction was obtained through a discussion with ChatGPT 5.6.
All mathematical statements and proofs were independently verified
by the author, who takes full responsibility for the contents.


\begin{thebibliography}{9}

\bibitem{balogh2023maximal}
J.~Balogh, C.~Chen, K.~Hendrey, B.~Lund, H.~Luo, C.~Tompkins, and T.~Tran.
\newblock Maximal 3-wise intersecting families.
\newblock {\em Combinatorica}, 43(6):1045--1066, 2023.

\bibitem{ellis2017symmetric}
D.~Ellis and B.~Narayanan.
\newblock On symmetric {$3$}-wise intersecting families.
\newblock {\em Proc. Amer. Math. Soc.}, 145(7):2843--2847, 2017.

\bibitem{erdos1974extremal}
P.~Erd\H{o}s and D.~J. Kleitman.
\newblock Extremal problems among subsets of a set.
\newblock {\em Discrete Math.}, 8(3):281--294, 1974.

\bibitem{frankl2016invitation}
P.~Frankl and N.~Tokushige.
\newblock Invitation to intersection problems for finite sets.
\newblock {\em J. Combin. Theory Ser. A}, 144:157--211, 2016.

\bibitem{frankston2018regular}
K.~Frankston, J.~Kahn, and B.~Narayanan.
\newblock On regular {$3$}-wise intersecting families.
\newblock {\em Proc. Amer. Math. Soc.}, 146(10):4091--4097, 2018.

\bibitem{janzer2022saturation}
B.~Janzer.
\newblock A note on saturation for {$k$}-wise intersecting families.
\newblock {\em Combinatorial Theory}, 2(2):Paper No. 11, 5, 2022.

\bibitem{oneill2021nontrivial}
J.~O'Neill and J.~Verstra{\"e}te.
\newblock Non-trivial {$d$}-wise intersecting families.
\newblock {\em J. Combin. Theory Ser. A}, 178:Paper No. 105369, 12, 2021.

\bibitem{tokushige2024measure}
N.~Tokushige.
\newblock The maximum measure of non-trivial {$3$}-wise intersecting families.
\newblock {\em Math. Program.}, 204(1--2):643--676, 2024.

\bibitem{zhang2023structure}
M.~Zhang and T.~Feng.
\newblock The structure of maximal non-trivial {$d$}-wise intersecting uniform
  families with large sizes.
\newblock {\em Discrete Math.}, 346(10):Paper No. 113533, 2023.

\end{thebibliography}
\end{document}